\documentclass[11pt]{amsart}

\usepackage[margin=1.15in]{geometry}
\usepackage{amsmath,amssymb,amsthm,amscd,mathtools}
\usepackage[hidelinks]{hyperref}

\theoremstyle{plain}
\newtheorem{theorem}{Theorem}[section]
\newtheorem{proposition}[theorem]{Proposition}
\newtheorem{lemma}[theorem]{Lemma}
\newtheorem{corollary}[theorem]{Corollary}
\theoremstyle{definition}
\newtheorem{definition}[theorem]{Definition}
\newtheorem{example}[theorem]{Example}
\theoremstyle{remark}
\newtheorem{remark}[theorem]{Remark}

\newcommand{\Q}{\mathbb{Q}}
\newcommand{\C}{\mathbb{C}}
\newcommand{\PP}{\mathbb{P}}
\newcommand{\Conf}{\operatorname{Conf}}
\newcommand{\Gr}{\operatorname{Gr}}
\newcommand{\Aut}{\operatorname{Aut}}
\newcommand{\Graphs}{\mathbf{Graph}}
\newcommand{\GrAlg}{\mathbf{GrAlg}_{\Q}}
\newcommand{\qbinom}[2]{\genfrac{[}{]}{0pt}{}{#1}{#2}}

\title{Edge-Span Chern Algebras of Graphical Configuration Spaces}
\author{Jackson Walters}
\address{Northern Virginia Community College, Virginia, USA}
\email{jwalters@nvcc.edu}
\date{July 21, 2026}
\hypersetup{
  pdftitle={Edge-Span Chern Algebras of Graphical Configuration Spaces},
  pdfauthor={Jackson Walters}
}

\begin{document}

\begin{abstract}
Place the vertices of a finite graph at projective points.  Each edge defines a
span map to $\Gr(2,n)$; the pulled-back Chern classes generate a graded algebra
$A_G^{(n)}$.  This assignment is a covariant graph functor, so the abstract
graded-algebra type is a graph invariant.  In ambient dimension four, the
Hilbert series is incomparable with the chromatic and Tutte polynomials.  On
five vertices in dimension three, the $34$ graph classes yield $33$ algebra
types, strictly refining both classical polynomials.  For every tree $T$, we
obtain a closed Hilbert-series formula depending only on $|V(T)|$ and $n$,
while $A_T^{(3)}$ determines $T$ up to isomorphism.  More precisely, its
cubic relation data is a complete tree invariant of polynomial size.  Finally,
the graphical configuration space embeds as a dense open in the picture
variety, and picture spaces with equal additive homology can have
nonisomorphic edge-Chern algebras.
\end{abstract}

\maketitle

\section{Introduction}

Graphical configuration spaces encode a graph topologically by forbidding
collisions precisely along its edges.  When the ambient space is projective,
they carry additional geometry: each edge determines a rank-two bundle given
by the span of its endpoints.  We ask whether the Chern classes of these
edge-span bundles assemble functorially, and how much graph structure their
multiplication retains.

The paper follows this question from construction to computation and then to
reconstruction.  The edge-span classes define a covariant functor
$G\mapsto A_G^{(n)}$ (Theorem~\ref{thm:main-functor}).  A graphical
K\v{r}{\'\i}\v{z}--Totaro model and its polynomial sector provide exact access to
these algebras (Theorem~\ref{thm:graphical-kriz-model} and
Proposition~\ref{prop:polynomial-sector}).  This leads to a census in which the
full algebra distinguishes $33$ of the $34$ graph classes on five vertices in
ambient dimension three (Theorem~\ref{thm:five-vertex-classification}), while
the Hilbert-series shadow is incomparable with the chromatic and Tutte
polynomials (Proposition~\ref{prop:classical-incomparability}).

Trees turn this computational evidence into a structural statement.  Their
Hilbert series depends only on the number of vertices and the ambient
dimension (Theorem~\ref{thm:tree-a-hilbert}), but $A_T^{(3)}$ determines $T$
(Theorem~\ref{thm:tree-reconstruction}).  In fact, the cubic relation pair
$(V_T,K_T)$ is already complete and admits a polynomial-size construction
(Corollary~\ref{cor:tree-cubic-invariant} and
Proposition~\ref{prop:tree-cubic-polynomial-size}).  This is a structural
result rather than a proposed tree-isomorphism algorithm: multiplication
retains the entire tree even though the Hilbert series forgets its shape.
Finally, the same span bundles extend across the collision boundary to graph
picture varieties (Proposition~\ref{prop:picture-compactification}).  The
constituent ideas have substantial precedents, compared in
Section~\ref{sec:related-constructions}; the contribution here is the
edge-span Chern subalgebra, its graph-functor structure, and the information
carried by its multiplication.

Let $G$ be a finite simple graph with vertex set $V(G)$ of size $r$.  Fix an
ambient vector space $\C^n$.  A projective point
\[
  \ell\in \PP^{n-1}(\C)
\]
is a one-dimensional subspace of $\C^n$, so it is natural to use projective
points as geometric vertices.  Two distinct projective points determine a
two-dimensional subspace,
\[
  (\ell_u,\ell_v)\longmapsto \langle \ell_u,\ell_v\rangle
  \in \Gr(2,n),
\]
which gives a geometric realization of the edge $uv$.

A first version of the construction places all $r$ vertices in the ordinary
ordered configuration space
\[
  \Conf_r(\PP^{n-1})
  =\{(\ell_1,\ldots,\ell_r):\ell_i\neq \ell_j\text{ for }i\neq j\}.
\]
Every pair then has a span, and a graph merely selects some of the resulting
span maps.  This is useful for a fixed labeled vertex set, but it is not
functorial for arbitrary graph homomorphisms: a homomorphism may identify two
nonadjacent vertices, while ordinary configuration spaces prohibit every
collision.

To allow such identifications, we forbid a collision only when the
corresponding pair is an edge.  The ambient configuration space then depends
on the graph.

\section{Graphical configuration spaces}

Throughout, $\Graphs$ denotes the category of finite simple loopless graphs and
graph homomorphisms.

\begin{definition}
Let $X$ be a Hausdorff space and let $G$ be a finite graph.  The \emph{graphical
configuration space} of $G$ in $X$ is
\[
  \Conf_G(X)
  =\left\{(x_v)_{v\in V(G)}\in X^{V(G)}:
    x_u\neq x_v\text{ whenever }uv\in E(G)\right\}.
\]
Equivalently,
\[
  \Conf_G(X)
  =X^{V(G)}\setminus\bigcup_{uv\in E(G)}\Delta_{uv},
  \qquad
  \Delta_{uv}=\{x_u=x_v\}.
\]
\end{definition}

Nonadjacent vertices are allowed to collide.  Thus the two extreme cases are
\[
  \Conf_{\overline K_r}(X)=X^r,
  \qquad
  \Conf_{K_r}(X)=\Conf_r(X).
\]

\begin{proposition}\label{prop:conf-functor}
For fixed $X$, graphical configuration space is a contravariant functor
\[
  \Conf_{(-)}(X):\Graphs^{\mathrm{op}}\longrightarrow \mathbf{Top}.
\]
\end{proposition}

\begin{proof}
Let $f:F\to G$ be a graph homomorphism.  Define
\[
  f_X:\Conf_G(X)\longrightarrow\Conf_F(X),
  \qquad
  (x_w)_{w\in V(G)}\longmapsto (x_{f(v)})_{v\in V(F)}.
\]
If $uv\in E(F)$, then $f(u)f(v)\in E(G)$, so
$x_{f(u)}\neq x_{f(v)}$.  Hence $f_X$ lands in $\Conf_F(X)$.  Directly from
the definition,
\[
  (g\circ f)_X=f_X\circ g_X
\]
and identity homomorphisms induce identity maps.
\end{proof}

The contravariance has a simple meaning: a $G$-configuration may be pulled
back along a vertex map $f:F\to G$ to produce an $F$-configuration.

\subsection{Colorings}

If $X$ is a finite set of cardinality $q$, then $\Conf_G(X)$ is exactly the set
of proper $X$-colorings of $G$.  Consequently,
\[
  |\Conf_G(X)|=P_G(q),
\]
where $P_G$ is the chromatic polynomial.  The same statement has a topological
shadow.  For a sufficiently well-behaved space $X$, additivity and
multiplicativity of compactly supported Euler characteristic give
\begin{equation}\label{eq:euler-chromatic}
  \chi_c(\Conf_G(X))=P_G(\chi_c(X)).
\end{equation}
In particular,
\[
  \chi_c\!\left(\Conf_G(\PP^{n-1}(\C))\right)=P_G(n).
\]

This is the Euler-characteristic shadow of deletion--contraction: for an edge
$e=uv$, the complement of $\Conf_G(X)$ in $\Conf_{G\setminus e}(X)$ is the
diagonal piece naturally identified with $\Conf_{G/e}(X)$; compare
Section~\ref{subsec:deletion-contraction}.  Thus graphical configuration
spaces are a topological refinement of graph coloring.

\section{Projective span maps}

We now specialize to
\[
  X_n=\PP^{n-1}(\C),\qquad n\geq 2.
\]
For every edge $e=uv\in E(G)$, the defining condition of $\Conf_G(X_n)$
ensures that $\ell_u$ and $\ell_v$ are distinct.  Their span is therefore
two-dimensional, and we obtain a regular map
\[
  \phi_e^G:\Conf_G(X_n)\longrightarrow\Gr(2,n),
  \qquad
  (\ell_w)_{w\in V(G)}\longmapsto\langle\ell_u,\ell_v\rangle.
\]

Let $\mathcal S$ denote the tautological rank-two subbundle on $\Gr(2,n)$.
For each vertex $v$, let $\mathcal L_v$ be the pullback to
$\Conf_G(X_n)$ of the tautological line bundle $\mathcal O(-1)$ from the
$v$th projective factor.  There is a canonical splitting
\[
  (\phi_e^G)^*\mathcal S\cong\mathcal L_u\oplus\mathcal L_v.
\]
If
\[
  h_v=c_1(\mathcal L_v^*)\in H^2(\Conf_G(X_n);\Q),
\]
then the two pulled-back Chern classes are
\begin{align*}
  a_e^G
  &=(\phi_e^G)^*c_1(\mathcal S^*)=h_u+h_v,\\
  b_e^G
  &=(\phi_e^G)^*c_2(\mathcal S^*)=h_uh_v.
\end{align*}
Indeed, dualizing the splitting above gives
\[
  c\bigl((\phi_e^G)^*\mathcal S^*\bigr)
  =c(\mathcal L_u^*)c(\mathcal L_v^*)
  =(1+h_u)(1+h_v),
\]
which also fixes the sign convention for these generators.
They have cohomological degrees $2$ and $4$, respectively.

\begin{definition}\label{def:edge-algebra}
Define
\[
  H_G^{(n)}
  =H^*(\Conf_G(\PP^{n-1}(\C));\Q)
\]
and let the \emph{edge-span algebra} be the unital graded subalgebra
\[
  A_G^{(n)}
  =\Q[a_e^G,b_e^G:e\in E(G)]
  \subseteq H_G^{(n)}.
\]
\end{definition}

The formulas $a_{uv}=h_u+h_v$ and $b_{uv}=h_uh_v$ are elementary, but their
products are taken in a cohomology ring that itself depends on the entire
diagonal arrangement determined by $G$.  The graph enters both through the
chosen edge classes and through the relations in the ambient ring.

\section{The cohomological graph functor}

Cohomology is contravariant.  Composing it with the contravariant construction
of Proposition~\ref{prop:conf-functor} reverses the arrows twice.

\begin{theorem}\label{thm:main-functor}
For every $n\geq2$, the assignment
\[
  G\longmapsto A_G^{(n)}
\]
extends canonically to a covariant functor
\[
  \mathcal A_n:\Graphs\longrightarrow\GrAlg.
\]
More precisely, $G\mapsto (A_G^{(n)}\subseteq H_G^{(n)})$ is a functor to
graded algebras equipped with a distinguished graded subalgebra.
\end{theorem}

\begin{proof}
Let $f:F\to G$ be a graph homomorphism.  Proposition~\ref{prop:conf-functor}
gives
\[
  f_{X_n}:\Conf_G(X_n)\longrightarrow\Conf_F(X_n),
\]
and hence a graded algebra map
\[
  f_{X_n}^*:H_F^{(n)}\longrightarrow H_G^{(n)}.
\]
For every edge $e=uv$ of $F$, the following diagram commutes:
\[
\begin{CD}
\Conf_G(X_n) @>{f_{X_n}}>> \Conf_F(X_n)\\
@V{\phi_{f(e)}^G}VV @VV{\phi_e^F}V\\
\Gr(2,n) @= \Gr(2,n).
\end{CD}
\]
The commutative square gives a canonical bundle isomorphism
\[
  f_{X_n}^*\bigl((\phi_e^F)^*\mathcal S^*\bigr)
  \cong (\phi_{f(e)}^G)^*\mathcal S^*.
\]
Taking Chern classes and using naturality therefore gives
\[
  f_{X_n}^*(a_e^F)=a_{f(e)}^G,
  \qquad
  f_{X_n}^*(b_e^F)=b_{f(e)}^G.
\]
Thus $f_{X_n}^*$ restricts to a homomorphism
\[
  \mathcal A_n(f):A_F^{(n)}\longrightarrow A_G^{(n)}.
\]
The identity and composition laws follow from those for the maps $f_{X_n}$
and for cohomological pullback.
\end{proof}

The full cohomology rings also form a functor
\[
  \mathcal H_n:G\longmapsto H_G^{(n)},
\]
and the inclusions $A_G^{(n)}\hookrightarrow H_G^{(n)}$ form a natural
transformation $\mathcal A_n\Rightarrow\mathcal H_n$.  We focus on the
distinguished subalgebra $A_G^{(n)}$.

\begin{proposition}[Disjoint unions]\label{prop:disjoint-unions}
For finite graphs $G$ and $H$, there are natural isomorphisms of graded
$\Q$-algebras
\[
  H_{G\sqcup H}^{(n)}
  \cong H_G^{(n)}\otimes_{\Q}H_H^{(n)},
  \qquad
  A_{G\sqcup H}^{(n)}
  \cong A_G^{(n)}\otimes_{\Q}A_H^{(n)}.
\]
With the empty graph as unit, these isomorphisms make $\mathcal A_n$ a strong
symmetric monoidal functor from disjoint union to tensor product.  In
particular, adjoining isolated vertices does not change the edge-span algebra.
\end{proposition}

\begin{proof}
No edge of $G\sqcup H$ joins a vertex of $G$ to a vertex of $H$, so the
defining noncollision conditions split and give a canonical product
decomposition
\[
  \Conf_{G\sqcup H}(X_n)
  \cong \Conf_G(X_n)\times\Conf_H(X_n).
\]
The rational K\"unneth theorem gives the cohomology isomorphism.  Under it,
the edge classes from $G$ are $a_e^G\otimes1,b_e^G\otimes1$, while those from
$H$ are $1\otimes a_e^H,1\otimes b_e^H$.  They generate the image of
$A_G^{(n)}\otimes A_H^{(n)}$, which embeds because tensoring over $\Q$
preserves the two inclusions into cohomology.  The monoidal compatibilities
follow from Cartesian product and K\"unneth, and $A_{K_1}^{(n)}=\Q$.
\end{proof}

\section{Relabeling and graph invariants}

Restricting the functor to graph isomorphisms immediately removes vertex
labels.
\begin{corollary}[Relabeling invariance]\label{cor:relabeling-invariance}
For every $\sigma\in S_r$, relabeling induces a canonical graded algebra
isomorphism
\[
  A_G^{(n)}\xrightarrow{\ \sim\ }A_{\sigma G}^{(n)}
\]
carrying $a_{ij}^G$ to $a_{\sigma(i)\sigma(j)}^{\sigma G}$ and $b_{ij}^G$ to
$b_{\sigma(i)\sigma(j)}^{\sigma G}$.  These isomorphisms are compatible with
composition in $S_r$.  Consequently,
\[
  G\longmapsto [A_G^{(n)}]
\]
is a graph invariant.
\end{corollary}

\begin{proof}
Apply Theorem~\ref{thm:main-functor} to the graph isomorphism defined by
$\sigma$; the formulas and their compatibility follow from functoriality.
\end{proof}

No quotient by $S_r$ is needed: passing to abstract algebra isomorphism type
already forgets the labeling.

\section{Functorial examples}

The following examples show how elementary graph maps appear under
$\mathcal A_n$.

\begin{example}[No edges and all edges]
For the empty graph on $r$ vertices,
\[
  \Conf_{\overline K_r}(X_n)=X_n^r,
  \qquad A_{\overline K_r}^{(n)}=\Q.
\]
For the complete graph,
\[
  \Conf_{K_r}(X_n)=\Conf_r(X_n),
\]
and $A_{K_r}^{(n)}$ is generated by the span classes for every pair.  Thus the
ordinary configuration-space construction is the complete-graph fiber of the
resulting functor.
\end{example}

\begin{example}[Adding edges]
If $F$ and $G$ have the same vertex set and $E(F)\subseteq E(G)$, the identity
on vertices is a graph homomorphism $F\to G$.  Geometrically,
\[
  \Conf_G(X_n)\hookrightarrow\Conf_F(X_n)
\]
because $G$ forbids more diagonals.  Cohomology gives
\[
  A_F^{(n)}\longrightarrow A_G^{(n)}.
\]
Thus the functor records the effect of adding edges through restriction to a
smaller diagonal complement.
\end{example}

\begin{example}[A noninjective homomorphism]
Let $P_3$ be the path $1-2-3$, and let $K_2$ have vertices $a,b$.  The map
\[
  f(1)=f(3)=a,\qquad f(2)=b
\]
is a graph homomorphism $P_3\to K_2$.  It induces
\[
  \Conf_{K_2}(X_n)\longrightarrow\Conf_{P_3}(X_n),
  \qquad
  (x_a,x_b)\longmapsto(x_a,x_b,x_a).
\]
This map is allowed precisely because vertices $1$ and $3$ are nonadjacent in
$P_3$.  On edge classes,
\[
  a_{12},a_{23}\longmapsto a_{ab},
  \qquad
  b_{12},b_{23}\longmapsto b_{ab}.
\]
The example exhibits the additional functoriality that was unavailable in the
ordinary configuration space $\Conf_3(X_n)$.
\end{example}

\subsection{Deletion and contraction}\label{subsec:deletion-contraction}

Let $e=uv$ be an edge of a simple graph $G$, and write

\[
  D=G\setminus e,
  \qquad
  C=G/e.
\]

There is generally no graph homomorphism $G\to C$ in the category of
loopless graphs, since such a map would identify the adjacent vertices $u$
and $v$.  After deleting $e$, however, they are nonadjacent.  The identity
map $D\to G$ and the quotient map $D\to C$ are graph homomorphisms, and they
produce a cospan
\[
  \Conf_G(X_n)\lhook\joinrel\longrightarrow \Conf_D(X_n)
  \longleftarrow \Conf_C(X_n).
\]
The right-hand map identifies $\Conf_C(X_n)$ with the closed diagonal
$\ell_u=\ell_v$ in $\Conf_D(X_n)$, while the left-hand map is the
complementary open inclusion.  Applying $\mathcal A_n$ gives
\[
  A_D^{(n)}\longrightarrow A_G^{(n)},
  \qquad
  A_D^{(n)}\longrightarrow A_C^{(n)}.
\]

For the smallest nontrivial case, take $G=K_3$ and $e=13$.  Then
\[
  D=P_3=(1-2-3),
  \qquad
  C=K_2.
\]
For $n=3$, put $x=a_{12}$, $y=a_{23}$, and let $z$ denote the degree-two
generator on the contracted edge.  Direct computation gives
\[
  A_{P_3}^{(3)}\cong \Q[x,y]/(x^3,y^3),
  \qquad
  A_{K_2}^{(3)}\cong \Q[z]/(z^3),
  \qquad |x|=|y|=|z|=2.
\]
The contraction-side map is the endpoint fold from the preceding example:
\[
  A_{P_3}^{(3)}\longrightarrow A_{K_2}^{(3)},
  \qquad x\longmapsto z,\quad y\longmapsto z,
\]
and it is surjective with kernel $(x-y)$.  The edge-addition map is
\[
  A_{P_3}^{(3)}\longrightarrow A_{K_3}^{(3)},
  \qquad x\longmapsto a_{12},\quad y\longmapsto a_{23}.
\]
Its kernel is $(xy(x+y))$.  Thus deletion and contraction are visible
functorially, but these maps do not automatically form a short exact sequence
of edge-span algebras.  For the full cohomology rings, the open--closed
decomposition instead belongs to the usual localization or Gysin formalism.

\section{Cohomology and a computational model}

To determine the algebras and their multiplication, we now pass from the
formal functor to a cohomological model.  The space $\Conf_G(X_n)$ is the
complement of the graphical diagonal
arrangement
\[
  \bigcup_{uv\in E(G)}\Delta_{uv}\subseteq X_n^{V(G)}.
\]
Its cohomology has been studied through deletion--contraction spectral
sequences and graph-dependent versions of the Bendersky--Gitler and
K\v{r}{\'\i}\v{z}--Totaro models; see
\cite{Kriz,Totaro,EastwoodHuggett,BaranovskySazdanovic,BokstedtMinuz}.

For projective space,
\[
  H^*(X_n;\Q)=\Q[h]/(h^n),\qquad |h|=2.
\]
Put
\[
  P_G^{(n)}
  =\Q[h_v:v\in V(G)]/(h_v^n:v\in V(G)).
\]
For each edge $e=uv$, introduce an odd generator $g_e=g_{uv}$ of degree
$2n-3$.  If $C=(e_1,\ldots,e_k)$ is a simple cycle, write
\[
  \partial(g_{e_1}\cdots g_{e_k})
  =\sum_{i=1}^k(-1)^{i-1}
    g_{e_1}\cdots\widehat{g_{e_i}}\cdots g_{e_k}.
\]
Define the graphical K\v{r}{\'\i}\v{z}--Totaro algebra
\begin{equation}\label{eq:graphical-kriz-algebra}
  \mathcal K_G^{(n)}
  =\frac{P_G^{(n)}\otimes
    \Lambda(g_e:e\in E(G))}{J_G},
\end{equation}
where $J_G$ is generated by
\begin{equation}\label{eq:graphical-kriz-relations}
  (h_u-h_v)g_{uv}
  \quad(uv\in E(G)),
  \qquad
  \partial(g_{e_1}\cdots g_{e_k})
  \quad(C\text{ a simple cycle}).
\end{equation}
The differential is determined by $dh_v=0$ and by the diagonal class
\begin{equation}\label{eq:diagonal-differential}
  d(g_{uv})=[\Delta_{uv}]
  =\sum_{k=0}^{n-1}h_u^{n-1-k}h_v^k.
\end{equation}

\begin{theorem}[Graphical K\v{r}{\'\i}\v{z}--Totaro model]
\label{thm:graphical-kriz-model}
Let $G$ be a finite simple graph and $n\geq2$.  There is an
$\Aut(G)$-equivariant zigzag of quasi-isomorphisms of commutative differential
graded algebras
\[
  \mathcal K_G^{(n)}
  \simeq A_{\mathrm{PL}}\!\left(\Conf_G(\PP^{n-1}(\C))\right),
\]
where $A_{\mathrm{PL}}$ denotes Sullivan's algebra of rational piecewise
polynomial forms.  In particular,
\[
  H^*(\mathcal K_G^{(n)})
  \cong H^*(\Conf_G(\PP^{n-1}(\C));\Q)
\]
as graded-commutative $\Q$-algebras.
\end{theorem}

\begin{proof}
Set $M=\PP^{n-1}(\C)$.  This is a smooth proper complex algebraic variety, and
Zakharov's chromatic configuration space $F(M,G)$ is precisely
$\Conf_G(M)$.  His chromatic K\v{r}{\'\i}\v{z}--Totaro theorem therefore applies
directly \cite[Theorem~6.3.1]{Zakharov}; see also \cite{ZakharovThesis}.  It
gives an $\Aut(G)$-equivariant zigzag from $A_{\mathrm{PL}}(\Conf_G(M))$ to a
cdga over $H^*(M^{V(G)};\Q)$ with one generator
$\widetilde\Delta_{uv}$ of degree $2\dim_{\C}M-1$ for each oriented edge.

We identify this cdga with \eqref{eq:graphical-kriz-algebra}.  First, the
K\"unneth isomorphism gives
\[
  H^*(M^{V(G)};\Q)
  =\Q[h_v:v\in V(G)]/(h_v^n:v\in V(G))
  =P_G^{(n)}.
\]
Since $\dim_{\C}M=n-1$, the edge generators have degree $2n-3$.  Zakharov
uses both orientations of an edge and imposes
$\widetilde\Delta_{uv}=\widetilde\Delta_{vu}$; choosing one orientation for
each undirected edge gives our generator $g_{uv}$.  Because
$\dim_{\mathbb R}M=2n-2$ is even, his cyclic relations are the usual graphical
Orlik--Solomon circuit-boundary relations
\[
  \partial(g_{e_1}\cdots g_{e_k})=0
\]
for simple cycles; see \cite[Theorem~6.2.1 and Remark~6.2.2]{Zakharov}.

The endpoint relations in Zakharov's presentation are
\[
  \bigl(p_u^*\gamma-p_v^*\gamma\bigr)g_{uv}=0
  \qquad(\gamma\in H^*(M;\Q)).
\]
Since $H^*(M;\Q)=\Q[h]/(h^n)$ is generated by $h$, this family generates the
same ideal as $(h_u-h_v)g_{uv}$: for every polynomial $p$,
$p(h_u)-p(h_v)$ is divisible by $h_u-h_v$.  Thus the endpoint and circuit
relations are exactly \eqref{eq:graphical-kriz-relations}.

Finally, Zakharov's differential sends $\widetilde\Delta_{uv}$ to the
Poincar\'e dual of the corresponding diagonal.  For projective space,
\[
  [\Delta]
  =\sum_{k=0}^{n-1}h^{n-1-k}\otimes h^k
  \in H^{2n-2}(M\times M;\Q),
\]
so this differential is \eqref{eq:diagonal-differential}.  All of the above
identifications intertwine the permutation action on vertices, and the
equivariant zigzag asserted by Zakharov therefore gives the claimed one.
\end{proof}

The graph-indexed algebra and its relation to the
Baranovsky--Sazdanovi\'c spectral sequence were studied earlier by
B\"okstedt and Minuz \cite{BokstedtMinuz}; in particular, the presentation
\eqref{eq:graphical-kriz-algebra} is not claimed as new here.  The specific
input supplied by Zakharov is the $\Aut(G)$-equivariant zigzag to
$A_{\mathrm{PL}}$ of the actual complement.  It upgrades the graph-indexed
complex to a multiplicative rational model and identifies its product with
the cup product on the graphical configuration space.

Fix a total order on $E(G)$, and let $\operatorname{NBC}(G)$ denote the
no-broken-circuit edge sets for the associated graphic matroid.  These edge
sets are forests.

\begin{proposition}[NBC normal form]\label{prop:nbc-normal-form}
If $F\in\operatorname{NBC}(G)$ and $\pi_0(F)$ is its set of connected
components, choose one root $\rho(C)$ in each component.  There is a graded
vector-space decomposition
\begin{equation}\label{eq:nbc-decomposition}
  \mathcal K_G^{(n)}
  \cong
  \bigoplus_{F\in\operatorname{NBC}(G)}
  \left(
    \bigotimes_{C\in\pi_0(F)}\Q[h_C]/(h_C^n)
  \right)g_F,
\end{equation}
where $|h_C|=2$ and $|g_F|=(2n-3)|F|$.  In particular, the elements
\begin{equation}\label{eq:nbc-normal-basis}
  \left(\prod_{C\in\pi_0(F)}h_{\rho(C)}^{m_C}\right)g_F,
  \qquad F\in\operatorname{NBC}(G),\quad 0\leq m_C<n,
\end{equation}
form a $\Q$-basis of $\mathcal K_G^{(n)}$.
\end{proposition}

\begin{proof}
The flatwise description underlying Zakharov's graphical model decomposes its
additive structure over the flats of the graphical diagonal arrangement; see
\cite[Section~6.2, especially Theorem~6.2.1]{Zakharov}.  A flat is determined
by a partition $\pi$ of $V(G)$ into connected blocks, and its polydiagonal is
\[
  \Delta_\pi\cong M^{\pi},
  \qquad M=\PP^{n-1}(\C).
\]
After the Thom isomorphism, its coefficient factor is
\[
  H^*(\Delta_\pi;\Q)
  \cong
  \bigotimes_{C\in\pi}\Q[h_C]/(h_C^n).
\]
The corresponding flat component of the graphical Orlik--Solomon algebra has,
by the no-broken-circuit theorem, basis the monomials $g_F$ for NBC forests
whose connected-component partition is $\pi$ \cite{OrlikSolomon}.  The Thom
shift $2(n-1)|F|$ together with the Orlik--Solomon degree $-|F|$ gives
$|g_F|=(2n-3)|F|$.

Under Zakharov's presentation, the endpoint relations identify $h_u$ and
$h_v$ on $g_F$ whenever $u$ and $v$ lie in the same component of $F$.
Consequently the factor indexed by a component $C$ may be represented by
powers of $h_{\rho(C)}$.  The flatwise direct sum and the NBC theorem give
linear independence, while the circuit and endpoint relations give spanning.
This proves both \eqref{eq:nbc-decomposition} and the displayed basis.
\end{proof}

This is the normal basis used by the full-DGA implementation; it is a basis of
$\mathcal K_G^{(n)}$ itself, rather than a further quasi-isomorphic
replacement.  The decomposition \eqref{eq:nbc-decomposition} is a
vector-space decomposition and is not asserted to respect multiplication.

The DGA has an auxiliary \emph{$g$-degree}, defined by giving every $h_v$
$g$-degree zero and every $g_e$ $g$-degree one.  All the relations in
\eqref{eq:graphical-kriz-relations} are homogeneous in this degree, so there is
a direct-sum decomposition
\[
  \mathcal K_G^{(n)}=\bigoplus_{q\geq0}\mathcal K_{G,q}^{(n)}.
\]
The differential maps $\mathcal K_{G,q}^{(n)}$ to
$\mathcal K_{G,q-1}^{(n)}$.  This separates the part needed for the edge-span
algebra from the rest of the cohomology.

\begin{proposition}[The polynomial sector]
\label{prop:polynomial-sector}
The inclusion $P_G^{(n)}\hookrightarrow\mathcal K_G^{(n)}$ induces an
injective homomorphism of graded algebras
\begin{equation}\label{eq:polynomial-sector}
  R_G^{(n)}
  :=\frac{\Q[h_v:v\in V(G)]}
  {(h_v^n,\Delta_{uv}:v\in V(G),\ uv\in E(G))}
  \lhook\joinrel\longrightarrow H^*(\mathcal K_G^{(n)}).
\end{equation}
Its image consists precisely of the cohomology classes represented by
elements of $g$-degree zero.
\end{proposition}

\begin{proof}
The $g$-degree-zero summand of $\mathcal K_G^{(n)}$ is $P_G^{(n)}$, and every
element of it is closed.  Because the differential lowers $g$-degree by one,
if an element of $P_G^{(n)}$ is a boundary, then it is the differential of the
$g$-degree-one component of a cochain.  Such a component can be represented as
\[
  \sum_{uv\in E(G)}f_{uv}(h)g_{uv},
\]
and its differential is
\[
  \sum_{uv\in E(G)}f_{uv}(h)\Delta_{uv}.
\]
This is well defined modulo the endpoint relation because
\[
  (h_u-h_v)\Delta_{uv}
  =h_u^n-h_v^n
  =0
  \qquad\text{in }P_G^{(n)}.
\]
Conversely, every multiple $f(h)\Delta_{uv}$ is the differential of
$f(h)g_{uv}$.  Hence
\[
  P_G^{(n)}\cap d\mathcal K_G^{(n)}
  =(\Delta_{uv}:uv\in E(G))\subseteq P_G^{(n)}.
\]
The kernel of the map from $P_G^{(n)}$ to cohomology is therefore exactly the
ideal generated by the diagonal classes.  Taking the indicated quotient gives
the asserted injection, and its image is precisely the set of classes having
a $g$-degree-zero representative.
\end{proof}

The classes defining $A_G^{(n)}$ are represented by the closed elements
\[
  a_{uv}=h_u+h_v,
  \qquad
  b_{uv}=h_uh_v.
\]
Proposition~\ref{prop:polynomial-sector} therefore gives the exact
computational description
\begin{equation}\label{eq:edge-algebra-in-rg}
  A_G^{(n)}
  =\Q[h_u+h_v,h_uh_v:uv\in E(G)]
  \subseteq R_G^{(n)}.
\end{equation}
In particular, all equalities and products among the edge-span classes may be
computed in the finite polynomial quotient $R_G^{(n)}$; no positive
$g$-degree cohomology is required.  The quotient $R_G^{(n)}$ is not, in
general, a model for the full configuration space: positive $g$-degree may
contribute additional cohomology classes.

The implementation retains three distinct levels of computation.  The full
graphical DGA is represented in a no-broken-circuit forest basis, with the
endpoint relations used to put the coefficient algebra into componentwise
normal form.  Exact differential matrices then compute all cohomology groups,
while multiplication by the polynomial classes gives the full multiplication
table of $A_G^{(n)}$.  For computations involving only $A_G^{(n)}$, the
program instead uses \eqref{eq:edge-algebra-in-rg} directly.  Finally, when
$G=K_r$, an independent reduced J-model for the ordinary configuration space
\cite{BerceanuMarklPapadima} supplies a check on both the cohomology groups and
the multiplication in $A_{K_r}^{(n)}$.  The reduced J-model is used for this
complete-graph comparison; it is not the model used for arbitrary $G$.

\section{The five-vertex computation}\label{sec:five-vertex-computation}

We first use the model to measure how much information the functor retains on
a complete finite census.  We computed $A_G^{(3)}$ for one representative of
each of the $34$ graph classes on five vertices.
Proposition~\ref{prop:polynomial-sector} gives each
algebra by closing the classes $a_e,b_e\in R_G^{(3)}$ under multiplication;
disconnected graphs were combined by tensor product using
Proposition~\ref{prop:disjoint-unions}.  The implementation
checks this calculation against the full graphical DGA on small examples and
against the reduced J-model for complete graphs.  Algebra comparisons forget
both the graph labels and the distinguished edge generators.

In ambient dimension three, the diagonal differential for an edge $e=uv$
gives
\[
  h_u^2+h_uh_v+h_v^2=0
\]
in cohomology.  Consequently,
\[
  (a_e)^2=(h_u+h_v)^2=h_uh_v=b_e,
\]
so $A_G^{(3)}$ is generated in cohomological degree two.  Regrading this as
ordinary degree one, let
\[
  \widetilde A_G^k=(A_G^{(3)})^{2k}.
\]
The multiplication table determines the canonical presentation
\begin{equation}\label{eq:canonical-presentation}
  \widetilde A_G
  \cong
  \operatorname{Sym}\bigl(\widetilde A_G^1\bigr)/I_G,
  \qquad
  (I_G)_k
  =\ker\!\left(
    \operatorname{Sym}^k\bigl(\widetilde A_G^1\bigr)
    \longrightarrow \widetilde A_G^k
  \right).
\end{equation}
This step is important: the polynomial ring in
\eqref{eq:canonical-presentation} is built intrinsically from the vector space
\(\widetilde A_G^1\), rather than from variables labeled by the edges of \(G\).

To compare the presentations intrinsically, we used the Hilbert series, the
socle Hilbert function, and the graded Betti table of the minimal free
resolution.  We also used the power-zero schemes
\[
  Z_p(A)=\{u\in A^1:u^p=0\}\subseteq A^1.
\]
Concretely, after choosing a basis $x_1,\ldots,x_d$ of $A^1$, we wrote
$u=t_1x_1+\cdots+t_dx_d$, reduced $u^p$ modulo $I_G$, and set all of its
coefficients equal to zero.  This produces a homogeneous ideal
$J_p\subseteq\Q[t_1,\ldots,t_d]$ defining $Z_p(A)$.  Although coordinates are
used to perform the calculation, a graded algebra isomorphism carries
$Z_p(A)$ linearly onto the corresponding scheme for the other algebra.
Therefore its dimension, degree, and graded Betti table are invariants of the
abstract graded algebra.

\begin{theorem}[Five-vertex classification]
\label{thm:five-vertex-classification}
As $G$ ranges over the $34$ isomorphism classes of simple graphs on five
vertices, the algebras $A_G^{(3)}$ assume exactly $33$ abstract graded-algebra
isomorphism types.  If $F$ and $G$ are nonisomorphic, then
\(A_F^{(3)}\cong A_G^{(3)}\) if and only if
\[
  \{F,G\}
  =\{P_3\sqcup 2K_1,\,2K_2\sqcup K_1\}.
\]
For this exceptional pair, both algebras are isomorphic to
\[
  \Q[x,y]/(x^3,y^3),
  \qquad |x|=|y|=2.
\]
\end{theorem}

\begin{proof}[Computational proof]
We enumerated the $2^{\binom{5}{2}}=1024$ labeled edge sets and minimized each
under all $120$ elements of $S_5$.  This gives $34$ canonical orbit
representatives.  For every representative, exact rational arithmetic gives
the multiplication table and hence the intrinsic presentation
\eqref{eq:canonical-presentation}.

The resulting Hilbert series assume $23$ values.  Sixteen Hilbert fibers are
singletons.  The remaining seven fibers contain $18$ graphs and were resolved
as shown below.  In the second column we record the Hilbert vector
\((\dim \widetilde A_G^k)_k\); \(\beta(Z_p)\) denotes the graded Betti table of
the homogeneous coordinate ring of \(Z_p\).
\begin{center}
\scriptsize
\begin{tabular}{ccccc}
\hline
\(|E(G)|\) & Hilbert vector & graphs & algebra types & separating data\\
\hline
2 & $(1,2,3,2,1)$ & 2 & 1 & identical presentations\\
3 & $(1,3,6,6,4,1)$ & 2 & 2 & $\beta(Z_3)$ differs\\
4 & $(1,4,10,13,11,5,1)$ & 3 & 3 & $\deg Z_5=14,12,10$\\
5 & $(1,5,10,10,5,1)$ & 4 & 4 & socle, $\deg Z_4$, and $\beta(Z_2)$\\
6 & $(1,5,9,7,2)$ & 3 & 3 & socle and $\beta(Z_2)$\\
7 & $(1,5,8,5,1)$ & 2 & 2 & socle Hilbert functions differ\\
7 & $(1,5,8,4)$ & 2 & 2 & $\deg Z_2=18,16$\\
\hline
\end{tabular}
\end{center}
Different entries among these intrinsic quantities obstruct a graded-algebra
isomorphism.  In the one remaining two-graph group, reducing the multiplication
tables to the presentations \eqref{eq:canonical-presentation} gives the same
relation space in the displayed bases, which supplies the explicit
isomorphism with \(\Q[x,y]/(x^3,y^3)\).  Thus the seven nonsingleton Hilbert
fibers contain $17$ algebra types.  Together with the $16$ singleton fibers,
this proves the asserted count and identifies the unique collision.
\end{proof}

The Python implementation uses exact rational arithmetic, SageMath~10.9 for
larger linear-algebra calculations, and Macaulay2~1.26.06 for minimal free
resolutions and power-zero ideals.  On the development machine, regenerating
all $34$ multiplication tables, presentations, and comparisons took about
$33$ seconds.  The source, regeneration command, and deterministic certificate
\path{results/r5_n3_classification.json} are available in the
\href{https://github.com/jacksonwalters/graph-functor}{\texttt{graph-functor}}
repository.  Release~\texttt{v1.0.3} is archived on Zenodo
\cite{WaltersGraphFunctorSoftware}.
The classification is exhaustive for $(r,n)=(5,3)$ only.
The same repository contains \path{tree_cubic_invariant.py}, which implements
the direct construction of the cubic tree invariant in
Proposition~\ref{prop:tree-cubic-polynomial-size}.  Regression tests compare
its output with the full-algebra presentation on small trees and verify
equivariance under vertex relabeling.

\subsection{Comparison with chromatic and Tutte polynomials}

Write
\[
  \mathsf h_n(G;t)=\operatorname{Hilb}(A_G^{(n)};t)
  =\sum_{k\geq0}\dim_{\Q}(A_G^{(n)})^k t^k.
\]
For two graph invariants $I$ and $J$, we say that $I$ \emph{refines} $J$ if
$I(F)=I(G)$ implies $J(F)=J(G)$.  The five-vertex computation distinguishes
sharply between the Hilbert-series shadow and the full algebra type.

\begin{proposition}\label{prop:classical-incomparability}
The invariant $\mathsf h_4(-;t)$ is incomparable with both the chromatic
polynomial and the Tutte polynomial.  On the other hand, when restricted to
simple graphs on five vertices, the abstract algebra type $[A_G^{(3)}]$
strictly refines both classical polynomials.
\end{proposition}

\begin{proof}
Let
\[
  F=P_3\sqcup2K_1,
  \qquad
  M=2K_2\sqcup K_1.
\]
Both are forests with two edges, and direct calculation gives
\[
  P_F(q)=P_M(q)=q^3(q-1)^2,
  \qquad
  T_F(x,y)=T_M(x,y)=x^2.
\]
Their ambient-dimension-four Hilbert series are different:
\begin{align*}
  \mathsf h_4(F;t)
  &={}
  1+2t^2+5t^4+6t^6+7t^8+5t^{10}+3t^{12}+t^{14},\\
  \mathsf h_4(M;t)
  &={}
  1+2t^2+5t^4+6t^6+8t^8+6t^{10}+5t^{12}+2t^{14}+t^{16}.
\end{align*}
Thus neither classical polynomial determines $\mathsf h_4$.

For the reverse direction, let $D$ be the graph obtained from a triangle by
attaching two leaves to the same triangle vertex.  Then
\[
  \mathsf h_4(D;t)=\mathsf h_4(C_5;t),
\]
while
\begin{align*}
  P_D(q)&=q(q-1)^3(q-2),
  &T_D(x,y)&=x^4+x^3+x^2y,\\
  P_{C_5}(q)&=(q-1)^5-(q-1),
  &T_{C_5}(x,y)&=x^4+x^3+x^2+x+y.
\end{align*}
Hence $\mathsf h_4$ determines neither classical polynomial.

Finally, Theorem~\ref{thm:five-vertex-classification} shows that the only
nontrivial fiber of $G\mapsto[A_G^{(3)}]$ on five-vertex graph classes is
$\{F,M\}$.  The displayed formulas show that both classical polynomials are
constant on this fiber.  The algebra type therefore refines each classical
polynomial on this census.  The refinement is strict because there are $33$
algebra types, but only $23$ chromatic-polynomial values and $25$
Tutte-polynomial values.
\end{proof}

For completeness, the full five-vertex census is summarized below.  A
\emph{colliding pair} is an unordered pair of nonisomorphic graphs on which
the indicated invariant agrees.
\begin{center}
\small
\begin{tabular}{lccc}
\hline
invariant & distinct values & colliding pairs & largest fiber\\
\hline
$P_G(q)$ & 23 & 15 & 3\\
$T_G(x,y)$ & 25 & 12 & 3\\
$\mathsf h_3(G;t)$ & 23 & 16 & 4\\
$\mathsf h_4(G;t)$ & 24 & 15 & 4\\
\hline
\end{tabular}
\end{center}
The joint invariants $(\mathsf h_4,P_G)$ and $(\mathsf h_4,T_G)$ each assume
$28$ values on the same $34$ graph classes.  The accompanying certificate
\path{results/r5_classical_comparison.json} records every graph, polynomial,
Hilbert series, collision fiber, and witness used in this comparison.

\begin{remark}
The reverse witness concerns only the Hilbert series:
Theorem~\ref{thm:five-vertex-classification} separates $D$ and $C_5$ as
abstract algebras.  It remains open whether $[A_G^{(n)}]$ determines either
classical polynomial for arbitrary graphs.
\end{remark}

\section{Tree families}\label{sec:tree-families}

The five-vertex census shows that multiplication can distinguish graphs whose
Hilbert series agree.  Trees isolate this contrast in a uniform family: neither
their chromatic nor their Tutte polynomial depends on shape.  Every tree $T$
on $r$ vertices has
\[
  P_T(q)=q(q-1)^{r-1},
  \qquad
  T_T(x,y)=x^{r-1}.
\]
Complete tree invariants are known.  Tree-specific polynomial constructions
of Chaudhary--Gordon and Liu recover rooted or unrooted trees
\cite{ChaudharyGordon,LiuTreePolynomial}, and Nishimura's second Kneser
chromatic function is also complete on trees \cite{NishimuraKneser}.  By
contrast, whether Stanley's chromatic symmetric function determines trees
remains open; for trees this is equivalent to the corresponding question for
the $U$-polynomial \cite{StanleyChromaticSymmetric,NobleWelsh,AlisteDeMierZamora}.

A closer topological precedent is Hoekstra--Mendoza's reconstruction of $T$
from the interaction complex governing the exterior face algebra
$H^*(UD_4(T))$ \cite{HoekstraMendoza}.  There the tree is the space on which
particles move; here it indexes forbidden diagonals and edge-span bundles in
projective space.  The result below is distinctive because the functorial
Chern algebra has a shape-independent Hilbert series, while its multiplication
still reconstructs the tree.

It is helpful first to isolate the part of the graphical model involving only
the projective hyperplane classes.  For a tree $T$, set
\[
  R_T^{(n)}=
  \frac{\Q[h_v:v\in V(T)]}
  {\bigl(h_v^n,\Delta_{uv}:uv\in E(T)\bigr)},
  \qquad
  \Delta_{uv}=\sum_{k=0}^{n-1}h_u^{n-1-k}h_v^k.
\]
Write $s=t^2$, so that $s$ records half the cohomological degree, and put
$[m]_s=1+s+\cdots+s^{m-1}$.

\begin{proposition}\label{prop:tree-r-hilbert}
If $T$ is a tree on $r$ vertices, then
\begin{equation}\label{eq:tree-r-hilbert}
  \operatorname{Hilb}\bigl(R_T^{(n)};s\bigr)
  =[n]_s[n-1]_s^{r-1}.
\end{equation}
\end{proposition}

\begin{proof}
Root $T$ at a vertex $\rho$, and write $p(v)$ for the parent of a nonroot
vertex $v$.  The ideal defining $R_T^{(n)}$ is generated by
\[
  h_\rho^n,
  \qquad
  \Delta_{v,p(v)}\quad(v\ne\rho).
\]
Indeed, the identity
\[
  h_v^n-h_{p(v)}^n
  =(h_v-h_{p(v)})\Delta_{v,p(v)}
\]
recovers all the omitted relations $h_v^n$ inductively from the root
relation.  Choose a lexicographic order in which every child variable is
larger than its parent.  The leading monomials are then
\[
  h_\rho^n,
  \qquad
  h_v^{n-1}\quad(v\ne\rho).
\]
They are pairwise relatively prime, so the displayed generators form a
Gr\"obner basis.  The standard monomials have root exponent less than $n$
and every other exponent less than $n-1$, which proves
\eqref{eq:tree-r-hilbert}.
\end{proof}

The passage from $R_T^{(n)}$ to $A_T^{(n)}$ is controlled by a conductor
calculation that works more generally for connected bipartite graphs.  Before
imposing the cohomological relations, put
\[
  P_G=\Q[h_v:v\in V(G)],
  \qquad
  S_G=\Q[h_u+h_v,h_uh_v:uv\in E(G)]\subseteq P_G.
\]

\begin{lemma}[Bipartite conductor lemma]\label{lem:bipartite-conductor}
Let $G$ be a connected bipartite graph with at least one edge and color
classes $V_0,V_1$.  Define
\[
  \pi_G:P_G\longrightarrow\Q[x,y],
  \qquad
  h_v\longmapsto
  \begin{cases}
    x,&v\in V_0,\\
    y,&v\in V_1.
  \end{cases}
\]
Then
\begin{equation}\label{eq:polynomial-bipartite-preimage}
  S_G=\pi_G^{-1}\bigl(\Q[x+y,xy]\bigr).
\end{equation}
\end{lemma}

\begin{proof}
Fix a vertex $u$.  The edge sums along paths express every $h_v$ in terms of
$S_G$ and $h_u$, so $P_G=S_G[h_u]$.  If $uw$ is any edge incident to $u$,
then
\[
  h_u^2-(h_u+h_w)h_u+h_uh_w=0,
\]
and multiplying this monic relation by $h_u^{k-2}$ gives, for $k\geq2$,
\[
  h_u^k=(h_u+h_w)h_u^{k-1}-h_uh_w\,h_u^{k-2}.
\]
Both coefficients on the right belong to $S_G$, so induction on $k$ reduces
every power of $h_u$ to an $S_G$-linear combination of $1$ and $h_u$.
Together with $P_G=S_G[h_u]$, this gives
\begin{equation}\label{eq:polynomial-rank-two}
  P_G=S_G+S_Gh_u.
\end{equation}

Consider a path of length two, $u-w-v$, and set $\delta=h_u-h_v$.  The edge
generators give
\[
  \delta=(h_u+h_w)-(h_w+h_v)\in S_G,
\]
\[
  h_w\delta=h_uh_w-h_wh_v\in S_G,
  \qquad
  h_u\delta=(h_u+h_w)\delta-h_w\delta\in S_G.
\]
Equation \eqref{eq:polynomial-rank-two} therefore implies
$\delta P_G\subseteq S_G$.  In other words, every difference arising from a
length-two path belongs to the conductor
\[
  \mathfrak c_G=\{f\in P_G:fP_G\subseteq S_G\}.
\]
Any two vertices in the same color class are joined by an even path, and
their difference is a sum of such length-two differences.  Consequently,
\[
  J_G=\ker\pi_G
  =(h_u-h_v:u,v\in V_0\text{ or }u,v\in V_1)
  \subseteq\mathfrak c_G\subseteq S_G.
\]

Certainly $\pi_G(S_G)=\Q[x+y,xy]$.  Conversely, if
$\pi_G(f)=F(x+y,xy)$, choose one edge $ab$ and set
$g=F(h_a+h_b,h_ah_b)$.  Then $g\in S_G$ and
$f-g\in J_G\subseteq S_G$, proving
\eqref{eq:polynomial-bipartite-preimage}.
\end{proof}

For a connected bipartite graph, let
\[
  I_G^{(n)}=(h_v^n,\Delta_{uv}:v\in V(G),\ uv\in E(G))\subseteq P_G
\]
and let $R_G^{(n)}=P_G/I_G^{(n)}$.  The image of $S_G$ in this quotient is
exactly the edge-span algebra $A_G^{(n)}$ by
Proposition~\ref{prop:polynomial-sector}.  The collapse
satisfies
\[
  \pi_G(I_G^{(n)})
  =(x^n,y^n,\Delta(x,y))=I_{K_2}^{(n)},
\]
so it induces a surjection
\[
  \overline\pi_G:R_G^{(n)}\longrightarrow R_{K_2}^{(n)}.
\]

\begin{corollary}\label{cor:bipartite-preimage}
For every connected bipartite graph $G$ with at least one edge,
\begin{equation}\label{eq:tree-bipartite-preimage}
  A_G^{(n)}
  =\overline\pi_G^{-1}\bigl(A_{K_2}^{(n)}\bigr).
\end{equation}
\end{corollary}

\begin{proof}
One inclusion follows because every edge generator maps to $x+y$ or $xy$.
For the converse, suppose the class of $f\in P_G$ maps into
$A_{K_2}^{(n)}$.  There are $g\in S_G$ and $q\in I_{K_2}^{(n)}$ such that
$\pi_G(f)=\pi_G(g)+q$.  Since
$\pi_G(I_G^{(n)})=I_{K_2}^{(n)}$, choose $p\in I_G^{(n)}$ with
$\pi_G(p)=q$.  Then
\[
  f-g-p\in\ker\pi_G=J_G\subseteq S_G
\]
by Lemma~\ref{lem:bipartite-conductor}.  Thus
$f\in S_G+I_G^{(n)}$, as required.
\end{proof}

\begin{theorem}\label{thm:tree-a-hilbert}
Let $T$ be a tree on $r\geq2$ vertices.  Then
\begin{equation}\label{eq:tree-a-hilbert}
  \operatorname{Hilb}\bigl(A_T^{(n)};s\bigr)
  =[n]_s[n-1]_s^{r-1}-[n]_s[n-1]_s+\qbinom{n}{2}_s,
\end{equation}
where
\[
  \qbinom{n}{2}_s
  =\frac{(1-s^n)(1-s^{n-1})}{(1-s)(1-s^2)}.
\]
\end{theorem}

\begin{proof}
Let $K=\ker\overline\pi_T$.  Corollary~\ref{cor:bipartite-preimage}
identifies $A_T^{(n)}$ as the full preimage of $A_{K_2}^{(n)}$.  Moreover, the
restriction $A_T^{(n)}\to A_{K_2}^{(n)}$ is surjective: the two generators on
any edge of $T$ map to the two generators on $K_2$.  Hence the same kernel $K$
occurs in two short exact sequences of graded vector spaces,
\[
  0\longrightarrow K\longrightarrow R_T^{(n)}
  \longrightarrow R_{K_2}^{(n)}\longrightarrow0
\]
and
\[
  0\longrightarrow K\longrightarrow A_T^{(n)}
  \longrightarrow A_{K_2}^{(n)}\longrightarrow0.
\]
The first two Hilbert series are
\[
  \operatorname{Hilb}(R_T^{(n)};s)=[n]_s[n-1]_s^{r-1},
  \qquad
  \operatorname{Hilb}(R_{K_2}^{(n)};s)=[n]_s[n-1]_s
\]
by Proposition~\ref{prop:tree-r-hilbert}.  For one edge, the span map fits
into the fiber bundle
\[
  \Conf_2(\PP^1)\longrightarrow
  \Conf_{K_2}(\PP^{n-1})\longrightarrow\Gr(2,n).
\]
The projection of $\Conf_2(\PP^1)$ onto its first coordinate has contractible
fiber $\PP^1\setminus\{\mathrm{point}\}\cong\C$, so
$H^*(\Conf_2(\PP^1);\Q)\cong\Q[u]/(u^2)$.  The classes $1$ and $h_1$ restrict
to the basis $1,u$ of this fiber cohomology.  Leray--Hirsch therefore makes the
pullback from $H^*(\Gr(2,n);\Q)$ injective.  Since that ring is generated by
the two Chern classes, its image is
$A_{K_2}^{(n)}$, and therefore
\[
  \operatorname{Hilb}(A_{K_2}^{(n)};s)=\qbinom{n}{2}_s.
\]
Subtracting the two exact-sequence identities proves
\eqref{eq:tree-a-hilbert}.
\end{proof}

In particular,
\[
  \dim_{\Q}A_T^{(n)}=n(n-1)^{r-1}-\binom n2.
\]

Although this Hilbert series is independent of the shape of $T$, multiplication
determines the tree.  The following result uses only $A_T^{(3)}$; no variation
of the ambient dimension is required.  Its proof will show more precisely that
cubic multiplication already suffices.

\begin{theorem}[Tree reconstruction]\label{thm:tree-reconstruction}
For finite trees $T$ and $T'$,
\[
  A_T^{(3)}\cong A_{T'}^{(3)}
  \quad\text{as graded $\Q$-algebras}
  \qquad\Longleftrightarrow\qquad
  T\cong T'.
\]
\end{theorem}

\begin{proof}
The implication from right to left follows from functoriality.  For the
converse, halve all cohomological degrees and write
\[
  z_{uv}=h_u+h_v\in A_T^1
  \qquad(uv\in E(T)).
\]
When $n=3$, the diagonal relation on an edge gives
\[
  h_u^2+h_uh_v+h_v^2=0,
  \qquad z_{uv}^2=h_uh_v,
  \qquad z_{uv}^3=0.
\]
Thus the $z_e$ generate $A_T^{(3)}$.  They are linearly independent: after
signing the rows according to the bipartition of $T$, their coefficient matrix
becomes an oriented incidence matrix of the tree, which has column rank
$|E(T)|$.  In particular, $\dim A_T^1=|E(T)|$.
There is only one tree with each of zero, one, or two edges, so those cases are
already determined by the abstract algebra.  We may henceforth assume that
$T$ has at least three edges.

We first recover the projective edge lines from the abstract algebra.  Extend
scalars to $\overline{\Q}$ and set
\[
  Z_T=\bigl\{[z]\in\PP(A_T^1):z^3=0\bigr\}_{\mathrm{red}}.
\]
We claim that
\begin{equation}\label{eq:tree-cube-zero-lines}
  Z_T=\bigl\{[z_e]:e\in E(T)\bigr\}.
\end{equation}

Let $V(T)=V_0\sqcup V_1$ be the bipartition and write
$z=\sum_e c_ez_e$.  For each vertex put
\[
  d_v=\sum_{e\ni v}c_e.
\]
Thus
\[
  z=\sum_{v\in V(T)}d_vh_v.
\]
Given a subset $S\subseteq V_1$, there is a graph homomorphism
$f_S:T\to P_3$ sending $V_0$ to the middle vertex and sending $S$ and its
complement to the two endpoints.  Functoriality gives
\[
  \begin{aligned}
    z&\longmapsto \alpha x+\beta y
      &&\text{in }A_{P_3}^{(3)}=\Q[x,y]/(x^3,y^3),\\
    \alpha&=\sum_{v\in S}d_v,
      &\qquad \beta&=\sum_{v\in V_1\setminus S}d_v.
  \end{aligned}
\]
Here $x^2y$ and $xy^2$ are linearly independent, and hence
\[
  (\alpha x+\beta y)^3=0
  \quad\Longleftrightarrow\quad
  \alpha\beta=0.
\]
Both color sums equal $C=\sum_e c_e$.  Taking $S=\{v\}$ gives
$d_v(C-d_v)=0$ for every $v\in V_1$.  If $C=0$, all these $d_v$ vanish; if
$C\ne0$, every $d_v$ is either $0$ or $C$, and their sum is $C$, so exactly
one is nonzero.  Reversing the color classes gives the same conclusion on
$V_0$.  In the first case the injectivity of the unsigned incidence matrix
gives $c_e=0$ for every $e$.  In the second case there are unique
$p\in V_0$ and $q\in V_1$ with $d_p=d_q=C$, while every other $d_v$ vanishes.
Therefore
\[
  z=C(h_p+h_q).
\]
Root $T$ at $p$.  The Gr\"obner basis from
Proposition~\ref{prop:tree-r-hilbert} has standard monomials with root
exponent at most two and every other exponent at most one.  If $pq$ is not
an edge, its distance from $p$ is at least three.  If $a$ is the parent of
$q$, then $a\ne p$ and
\[
  (h_p+h_q)^3
  =3h_p^2h_q-3h_ph_qh_a-3h_ph_a^2.
\]
The first two terms are standard, while reducing the last term introduces no
factor $h_q$.  Thus the normal form retains $3h_p^2h_q$ and is nonzero.  If
$pq$ is an edge, the cube vanishes by the edge relation.  This proves
\eqref{eq:tree-cube-zero-lines}.

It remains to reconstruct incidence among these recovered edge lines.  For
three distinct points $[x],[y],[z]\in Z_T$, let $U=\langle x,y,z\rangle$ and
define the intrinsic number
\[
  \kappa(U)=\dim\ker\bigl(\operatorname{Sym}^3U\longrightarrow A_T^3\bigr).
\]
The three pure cubes always lie in this kernel.  Reduction by the same rooted
Gr\"obner basis gives the following possibilities; the displayed relation is
the unique additional relation modulo $x^3,y^3,z^3$.
\[
\begin{array}{c|c|c}
\text{selected edges}&\kappa(U)&\text{additional cubic relation}\\ \hline
\text{disconnected}&3&\text{none}\\
K_{1,3}&4&(x-y)(x-z)(y-z)\\
P_4\text{, with central edge }x&4&
  -xy^2-xyz-xz^2+y^2z+yz^2
\end{array}
\]
We justify every entry in the table.  If the three edges form a connected
subtree $S$, then $S$ is either $K_{1,3}$ or $P_4$ and is a graph retract of
$T$.  Explicitly, every component of $T\setminus V(S)$ attaches to a unique
vertex $v$ of $S$.  Choose a neighbor $w_v$ of $v$ in $S$ and map that
component alternately to $v$ and $w_v$ according to distance from $v$ (even
distance to $v$ and odd distance to $w_v$), while fixing $S$ pointwise.  This
defines a loopless graph homomorphism
$T\to S$.  Functoriality therefore makes $A_S^{(3)}\to A_T^{(3)}$ split
injective, with image the subalgebra generated by $x,y,z$.

The coefficient of $s^3$ in Theorem~\ref{thm:tree-a-hilbert} for a four-vertex
tree and $n=3$ is $6$.  Since $\dim\operatorname{Sym}^3U=10$, this gives
$\kappa(U)=4$.  The extra relations in the table can be checked directly.
For the star, write $c$ for the central vertex class and $a,b,d$ for the leaf
classes, so $x=c+a$, $y=c+b$, and $z=c+d$.  Then
\[
  (x-y)(x-z)(y-z)
  =(b-d)\Delta_{ca}+(d-a)\Delta_{cb}+(a-b)\Delta_{cd}=0.
\]
For the path $a-b-c-d$, write the vertex names also for their $h$-classes and
put $y=a+b$, $x=b+c$, and $z=c+d$.  The displayed path relation $F$ satisfies
\[
  F=(d-a)\Delta_{ab}-(a+b+d)\Delta_{bc}
    +(a-c)\Delta_{cd}+a^3=0.
\]
Neither relation lies in the span of the three pure cubes, so the dimension
count shows that it is the unique additional relation.

If the selected edges are disconnected, they are either a matching or an
adjacent pair together with a disjoint edge.  For a matching, the seven mixed
cubic monomials have pairwise disjoint supports in the squarefree standard
monomial basis, using $z_{uv}^2=h_uh_v$.  In the second case, write the
adjacent edges as $vu,vw$, the disjoint edge as $ab$, and root $T$ at $v$.
Put $x=z_{vu}$, $y=z_{vw}$, and $z=z_{ab}$.  All terms below are standard.
Successively inspect
\[
  h_v^2h_u,\quad h_v^2h_w,\quad h_uh_ah_b,\quad h_wh_ah_b,
  \quad h_v^2h_a,\quad h_uh_vh_a,\quad h_wh_vh_a
\]
in $x^2y,xy^2,xz^2,yz^2,xyz,x^2z,y^2z$, respectively.  Eliminating the
coefficients in this order proves that the seven cubics are linearly
independent.  Hence $\kappa(U)=3$ in both disconnected cases.

The two connected cases are intrinsically different.  The star relation has
no $xyz$ term and every variable occurs squared.  The path relation has a
nonzero $xyz$ term, and its central edge $x$ is the unique variable that never
occurs squared.  These properties are unchanged by rescaling or permuting the
three recovered lines.  We can therefore determine exactly which pairs of
edges of $T$ share a vertex, and hence reconstruct the line graph $L(T)$.
For trees with at least three edges every adjacent pair occurs in such a
connected triple; the smaller trees are determined by their number of edges.
Finally, Whitney's line-graph theorem \cite{Whitney} reconstructs $T$ from
$L(T)$.  Its only connected exceptional pair is $K_3$ and $K_{1,3}$, and only
the latter is a tree.  Every step used only the abstract graded multiplication,
so a graded algebra isomorphism $A_T^{(3)}\cong A_{T'}^{(3)}$ yields
$T\cong T'$.  This completes the reconstruction.
\end{proof}

\begin{corollary}[Cubic tree invariant]\label{cor:tree-cubic-invariant}
For a finite tree $T$, set
\[
  V_T=\widetilde A_T^1,
  \qquad
  K_T=\ker\!\left(
    \operatorname{Sym}^3(V_T)\longrightarrow \widetilde A_T^3
  \right).
\]
Then $T\cong T'$ if and only if there is a linear isomorphism
$f:V_T\to V_{T'}$ such that
\[
  \operatorname{Sym}^3(f)(K_T)=K_{T'}.
\]
Thus the cubic relation data $(V_T,K_T)$ is a complete invariant of finite
trees.
\end{corollary}

\begin{proof}
A tree isomorphism induces the required linear isomorphism by functoriality.
Conversely, the proof of Theorem~\ref{thm:tree-reconstruction} uses only
$V_T$ and $K_T$: the pure cubes in $K_T$ recover the projective edge lines,
and the intersections $K_T\cap\operatorname{Sym}^3(U)$ for their
three-dimensional spans $U$ recover edge incidence.  The resulting line graph
determines $T$.
\end{proof}

\begin{proposition}[Polynomial-size cubic invariant]
\label{prop:tree-cubic-polynomial-size}
Given a tree $T$ on $r$ vertices, the cubic relation data $(V_T,K_T)$ can be
represented using $O(r^6)$ field entries and constructed using polynomially
many arithmetic operations over $\Q$.  For $r\geq2$, this contrasts with
\[
  \dim_{\Q}A_T^{(3)}=3\cdot2^{r-1}-3.
\]
\end{proposition}

\begin{proof}
The case $r=1$ is immediate: both $V_T$ and $K_T$ vanish.  Assume $r\geq2$.
Theorem~\ref{thm:tree-a-hilbert} specializes to
\[
  \operatorname{Hilb}(A_T^{(3)};s)
  =(1+s+s^2)\bigl((1+s)^{r-1}-s\bigr),
\]
so every fixed-degree piece through degree three has polynomial dimension.
We give a direct construction of the multiplication kernel.

Give each $h_v$ degree one and let $P_3$ be the degree-three part of
$\Q[h_v:v\in V(T)]$.  It has dimension $\binom{r+2}{3}$.  By
Proposition~\ref{prop:polynomial-sector}, the degree-three relations in
$R_T^{(3)}$ are spanned by
\[
  h_v^3\quad(v\in V(T)),
  \qquad
  h_w\Delta_{uv}\quad(w\in V(T),\ uv\in E(T)),
\]
where $\Delta_{uv}=h_u^2+h_uh_v+h_v^2$.  Thus the quotient
$(R_T^{(3)})_3$ is obtained by row-reducing a matrix with $O(r^2)$ rows and
$O(r^3)$ columns.

The incidence-matrix argument in the proof of
Theorem~\ref{thm:tree-reconstruction} shows that the edge sums
$z_{uv}=h_u+h_v$ form a basis of $V_T$.  The map
\[
  \operatorname{Sym}^3(V_T)\longrightarrow (R_T^{(3)})_3
\]
sends a product of three basis elements to a polynomial with at most eight
terms.  Its source has dimension $\binom{r+1}{3}$, and its kernel is exactly
$K_T$ because $A_T^{(3)}=\Q[z_{uv}:uv\in E(T)]$ embeds in $R_T^{(3)}$.
Constructing the matrix and computing its kernel therefore uses polynomially
many arithmetic operations.  A row-space basis for $K_T$ has at most
$O(r^6)$ entries.
\end{proof}

\begin{remark}[Recognition versus construction]
\label{rem:tree-cubic-algorithmic-scope}
The proposition constructs $(V_T,K_T)$ from a labeled tree; it does not address
the complexity of recognizing equivalence of two such pairs under
$\operatorname{GL}(V_T)$.
\end{remark}

\begin{remark}\label{rem:tree-reconstruction-dimension}
The proof is genuinely specific to ambient dimension three.  It recovers the
edge generators as the reduced projective cube-zero locus and then reads their
incidence from cubic multiplication.  Neither step follows formally from the
general conductor description.  Whether $A_T^{(n)}$ determines $T$ for every
$n\geq 4$ remains open.
\end{remark}

For $(r,n)=(6,3)$, all six unlabeled trees have the common Hilbert series
\[
  1+5t^2+15t^4+24t^6+25t^8+16t^{10}+6t^{12}+t^{14}.
\]
Their abstract algebras are nevertheless pairwise nonisomorphic.  Indeed, for
the intrinsic power-zero schemes introduced above, the six values of
$\deg Z_5(A_T^{(3)})$ are
\[
  16,\ 18,\ 20,\ 22,\ 24,\ 30.
\]
Since these degrees are all different, they separate the six algebra types.
They provide independent computational certificates for the six algebra
types and illustrate the distinction between shape-blind Hilbert series and
the complete multiplication invariant.

\section{Graph picture spaces}\label{sec:picture-spaces}

Thus far the edge-span algebra has lived on the open graphical configuration
space.  The same span maps place it inside the projective incidence geometry
of graph picture spaces introduced by Martin
\cite{MartinGeometry,MartinTopology}.  For
$d=n-1$, the \emph{picture space} of $G$ is the projective incidence variety
\[
  \begin{aligned}
  \mathcal X^d(G)=\bigl\{((\ell_v),(P_e))\in{}&
    (\PP^d)^{V(G)}\times\Gr(2,n)^{E(G)}:\\
    &\ell_u,\ell_v\subseteq P_{uv}
      \text{ for every }uv\in E(G)\bigr\}.
  \end{aligned}
\]
Adjacent vertices are allowed to coincide in $\mathcal X^d(G)$; when they do,
the plane $P_{uv}$ is additional data rather than the span of two distinct
points.  The \emph{picture variety} $\mathcal V^d(G)$ is the closure in
$\mathcal X^d(G)$ of the locus on which all vertex points are distinct.  The
full picture space can be reducible, and $\mathcal V^d(G)$ is its distinguished
near-generic component.

\begin{proposition}\label{prop:picture-compactification}
The map
\[
  j_G:\Conf_G(\PP^{n-1})\longrightarrow\mathcal X^{n-1}(G),
  \qquad
  (\ell_v)\longmapsto
  ((\ell_v),(\langle\ell_u,\ell_v\rangle)_{uv\in E(G)}),
\]
is an open immersion with image the locus on which the endpoints of every
edge are distinct.  Its image lies in $\mathcal V^{n-1}(G)$ and is dense
there.  In particular,
\[
  \overline{j_G(\Conf_G(\PP^{n-1}))}=\mathcal V^{n-1}(G).
\]
\end{proposition}

\begin{proof}
On the indicated open locus, the plane belonging to an edge is uniquely the
span of its endpoints, so projection onto the vertex coordinates is inverse
to $j_G$.  The locus is open because the inequalities
$\ell_u\ne\ell_v$ are open conditions.  The all-distinct vertex locus is
dense in the graphical configuration space, which is itself a nonempty open
subset of the irreducible variety $(\PP^{n-1})^{V(G)}$.  Its image is the
generic-picture locus used to define $\mathcal V^{n-1}(G)$, so the two
closures agree.
\end{proof}

For each edge, let
\[
  \rho_e:\mathcal V^{n-1}(G)\longrightarrow\Gr(2,n)
\]
be the projection recording its plane, and put
\[
  \mathcal S_e=\rho_e^*\mathcal S.
\]
Define the \emph{picture edge algebra}
\[
  B_G^{(n)}
  =\Q[c_1(\mathcal S_e^*),c_2(\mathcal S_e^*):e\in E(G)]
  \subseteq H^*(\mathcal V^{n-1}(G);\Q).
\]
Since $\rho_e\circ j_G=\phi_e^G$, restriction gives a canonical surjection
\begin{equation}\label{eq:picture-to-open-edge-algebra}
  j_G^*:B_G^{(n)}\twoheadrightarrow A_G^{(n)}.
\end{equation}
Thus $A_G^{(n)}$ is the edge-Chern algebra of the picture variety after
removing its collision boundary.

\begin{remark}[Relation with Martin's Chern generators]
The individual edge classes above already appear naturally in Martin's
description of orchard picture spaces \cite[Theorem~17 and \S7.2]{MartinTopology}.
For a vertex $v$ and an incident edge $e=vw$, let $\mathcal L_v$ be the
line bundle with fiber $P(v)$ and let $\mathcal K_{e,v}$ have fiber
$P(e)/P(v)$.  The edge-plane bundle fits into
\[
  0\longrightarrow \mathcal L_v\longrightarrow \mathcal S_e
  \longrightarrow \mathcal K_{e,v}\longrightarrow0.
\]
Writing
\[
  x_v=c_1(\mathcal L_v^*),\qquad
  y_{e,v}=c_1(\mathcal K_{e,v}^*),
\]
the Whitney formula gives
\[
  c_1(\mathcal S_e^*)=x_v+y_{e,v},\qquad
  c_2(\mathcal S_e^*)=x_vy_{e,v}.
\]
Martin denotes the symmetric first class by
$z_e=x_v+y_{e,v}=x_w+y_{e,w}$.  On the open graphical configuration locus,
$P(e)/P(v)$ is canonically identified with the other endpoint line $P(w)$;
hence
\[
  (x_v,y_{e,v})\longmapsto(h_v,h_w),
\]
and Martin's two Chern expressions restrict exactly to $a_e$ and $b_e$.
Thus neither the individual classes nor functoriality alone is new: Martin's
picture spaces are likewise contravariantly functorial for
incidence-preserving maps \cite[Remark~8]{MartinTopology}.  Here we study the
edge-generated algebra on the open locus.  Proposition~\ref{prop:picture-tutte-separation}
shows that additive picture-space homology determines neither this algebra nor
its picture analogue even at the level of dimension ($30$ versus $36$), while
Theorem~\ref{thm:tree-reconstruction} shows that multiplication reconstructs a
tree when the Hilbert series does not see its shape.
\end{remark}

The smallest nontrivial example makes this comparison explicit.  Let
$P_3$ have edges $12$ and $23$.  Write $V=\C^n$, and on $\PP(V)$ let
$\mathcal L$ be the tautological line and $\mathcal Q=V/\mathcal L$.  Using
the convention that $\PP(\mathcal Q)$ parametrizes lines in $\mathcal Q$, set
\[
  Y=\PP(\mathcal Q)\times_{\PP(V)}\PP(\mathcal Q).
\]
Let $\mathcal M_1,\mathcal M_3$ be the two tautological lines on $Y$, and let
$\mathcal S_1,\mathcal S_3$ be their inverse images in the trivial bundle
$V\otimes\mathcal O_Y$.  Thus
\[
  0\longrightarrow\mathcal L\longrightarrow\mathcal S_i
  \longrightarrow\mathcal M_i\longrightarrow0,
  \qquad i=1,3.
\]
Choosing the two remaining vertex points gives an iterated projective-bundle
description
\begin{equation}\label{eq:p3-picture-bundle}
  \mathcal X^{n-1}(P_3)
  \cong\PP(\mathcal S_1)\times_Y\PP(\mathcal S_3).
\end{equation}
In particular, this picture space is smooth and irreducible, so it equals its
picture variety, and the pullback from $H^*(Y;\Q)$ is injective.

Put
\[
  h=c_1(\mathcal L^*),\qquad
  \xi_i=c_1(\mathcal M_i^*)\quad(i=1,3),
\]
and write
\[
  \Delta(z,w)=\sum_{k=0}^{n-1}z^{n-1-k}w^k.
\]
The projective-bundle formula gives
\begin{equation}\label{eq:p3-picture-base-ring}
  H^*(Y;\Q)
  \cong
  \frac{\Q[h,\xi_1,\xi_3]}
  {(h^n,\Delta(h,\xi_1),\Delta(h,\xi_3))}.
\end{equation}
The Chern classes of the two edge-plane bundles are
\[
  h+\xi_1,\ h\xi_1,
  \qquad
  h+\xi_3,\ h\xi_3.
\]
On the graphical configuration-space locus, the quotient line
$\mathcal S_i/\mathcal L$ is identified with the corresponding endpoint line,
so restriction sends
\[
  (\xi_1,h,\xi_3)\longmapsto(h_1,h_2,h_3).
\]
The presentation \eqref{eq:p3-picture-base-ring} is precisely
$R_{P_3}^{(n)}$: the omitted relations $\xi_i^n=0$ follow from
$(h-\xi_i)\Delta(h,\xi_i)=h^n-\xi_i^n$.  Proposition~\ref{prop:polynomial-sector}
therefore shows that restriction is injective on $H^*(Y;\Q)$.  We have proved
\begin{equation}\label{eq:p3-picture-edge-isomorphism}
  B_{P_3}^{(n)}\xrightarrow{\ \sim\ }A_{P_3}^{(n)}.
\end{equation}

For one edge, the picture space is
\[
  \mathcal X^{n-1}(K_2)
  =\PP(\mathcal S)\times_{\Gr(2,n)}\PP(\mathcal S).
\]
The projective-bundle theorem makes the pullback of
$H^*(\Gr(2,n);\Q)$ injective, and the Leray--Hirsch argument in the proof of
Theorem~\ref{thm:tree-a-hilbert} does the same after restriction to
$\Conf_{K_2}(\PP^{n-1})$.  Hence
\begin{equation}\label{eq:k2-picture-edge-isomorphism}
  B_{K_2}^{(n)}\xrightarrow{\ \sim\ }A_{K_2}^{(n)}.
\end{equation}

We can now compare the first pair relevant to the Tutte polynomial.  Set
\[
  G_1=P_3\sqcup K_1,
  \qquad
  G_2=K_2\sqcup K_2.
\]
\begin{proposition}\label{prop:picture-tutte-separation}
In every ambient dimension, the picture spaces of $G_1$ and $G_2$ have
isomorphic additive rational homology.  In ambient dimension four, their
picture edge algebras are nonisomorphic; indeed,
\[
  \dim_{\Q}B_{G_1}^{(4)}=30,
  \qquad
  \dim_{\Q}B_{G_2}^{(4)}=36.
\]
Consequently, neither the picture edge algebra nor the open edge-span algebra
is determined by the Tutte polynomial or by the additive homology of the
picture space.
\end{proposition}

\begin{proof}
Both graphs have four vertices and
\[
  P_{G_1}(q)=P_{G_2}(q)=q^2(q-1)^2,
  \qquad
  T_{G_1}(x,y)=T_{G_2}(x,y)=x^2.
\]
For $n\geq3$, Martin's formula \cite[Theorem~1]{MartinTopology} therefore gives
isomorphic additive homology groups for their picture spaces in every fixed
ambient dimension.  When $n=2$, every edge-plane is the whole of $\C^2$, so
both picture spaces are simply $(\PP^1)^4$ and the same conclusion holds.
Picture varieties split as products under disjoint unions, and their edge
bundles pull back from the corresponding factors; the K\"unneth argument of
Proposition~\ref{prop:disjoint-unions} therefore shows that their picture edge
algebras tensor, while isolated vertices contribute no edge generators.
Equations
\eqref{eq:p3-picture-edge-isomorphism} and
\eqref{eq:k2-picture-edge-isomorphism} identify the picture edge algebras
with the corresponding open edge-span algebras.  For $n=4$, writing $s=t^2$,
Theorem~\ref{thm:tree-a-hilbert} gives
\begin{align*}
  \operatorname{Hilb}(B_{G_1}^{(4)};s)
  &=1+2s+5s^2+6s^3+7s^4+5s^5+3s^6+s^7,\\
  \operatorname{Hilb}(B_{G_2}^{(4)};s)
  &=\left(\qbinom{4}{2}_s\right)^2\\
  &=1+2s+5s^2+6s^3+8s^4+6s^5+5s^6+2s^7+s^8.
\end{align*}
Their dimensions are $30$ and $36$, respectively.  Consequently, neither
$B_G^{(4)}$ nor $A_G^{(4)}$ is determined by the Tutte polynomial or by the
additive homology of the picture space.  This does not assert that the
undecorated cohomology rings of the two picture spaces are different; the
additional information is carried by the distinguished edge Chern classes
and their multiplication.
\end{proof}

\section{Comparison with related constructions}
\label{sec:related-constructions}

The preceding section places the functor in its projective-geometric setting.
We now compare it with other algebraic and cohomological constructions attached
to graphs.  The classical starting point is the graphic hyperplane arrangement.
If
$X=\C$, then
\[
  \Conf_G(\C)
  =\C^{V(G)}\setminus\bigcup_{uv\in E(G)}\{z_u=z_v\}
\]
is the complement of the graphic arrangement.  Its cohomology is the
Orlik--Solomon algebra of the graphic matroid
\cite{OrlikSolomon}.  The projective construction considered here replaces
hyperplane-arrangement classes by diagonal-complement cohomology and adds the
geometric span maps to Grassmannians.

Eastwood and Huggett used these spaces to express chromatic-polynomial values
as Euler characteristics \cite{EastwoodHuggett}.
Their homology was subsequently approached through graph cohomology and
graph-dependent K\v{r}{\'\i}\v{z}--Totaro models
\cite{BaranovskySazdanovic,BokstedtMinuz}.  Thus both the spaces and their
cohomological models are established.

\subsection{Recent graphical-configuration constructions}

Khoroshkin and Lyskov study the same underlying spaces $\Conf_G(X)$, including
$X=\PP^{n-1}(\C)$, through twisted Lie algebras and contractad operations on
their compactifications \cite[Theorem~2.7]{KhoroshkinLyskov}.  Their first
Chern classes come from vertex- and tube-indexed line bundles on wonderful
compactifications of $\Conf_G(\C)/\operatorname{Aff}(\C)$
\cite[Section~5.1]{KhoroshkinLyskov}.  Our classes are instead $c_1$ and $c_2$
of rank-two bundles pulled back along edge-span maps on the open projective
space, with ambient dimension retained.  For example, $A_G^{(2)}=\Q$ because
$\Gr(2,2)$ is a point, so these classes cannot agree with their $\psi$-classes.

Crowley, Dorpalen-Barry, Henriques, and Proudfoot consider another graphical
configuration space
\[
  X_{\mathrm{CDHP}}(G)
  =\Conf_G(SU(2))/SU(2)^{\pi_0(G)},
\]
where each copy of $SU(2)$ acts by left translation on one connected
component.  They identify its integral cohomology with an internal zonotopal
algebra of the cographical vector arrangement by a canonical degree-halving
isomorphism \cite[Theorem~1.9]{CrowleyEtAl}.  The following elementary
comparison shows that this algebra and the edge-span algebra retain different
graph information.

\begin{proposition}\label{prop:crowley-tree-comparison}
For every finite tree $T$,
\[
  H^*(X_{\mathrm{CDHP}}(T);\Q)
  \cong \operatorname{IZ}(T;\Q)
  \cong \Q.
\]
By contrast, $A_T^{(3)}$ determines $T$ up to graph isomorphism.  Hence the two
algebra-valued constructions do not agree, even after restriction to trees.
\end{proposition}

\begin{proof}
Choose a root of $T$ and orient every edge away from it.  Left translation
gives each point of $X_{\mathrm{CDHP}}(T)$ a unique representative whose root
coordinate is the identity.  Sending such a representative $(x_v)$ to
\[
  (x_u^{-1}x_v)_{u\to v\in E(T)}
\]
defines a diffeomorphism
\[
  X_{\mathrm{CDHP}}(T)
  \cong (SU(2)\setminus\{1\})^{E(T)}.
\]
Indeed, the unique paths from the root reconstruct all vertex coordinates
from the edge coordinates.  Since $SU(2)\cong S^3$, each factor is
contractible.  Equivalently, the cographical arrangement takes values in
$H^1(T;\mathbb{Z})=0$, so its internal zonotopal algebra is $\Q$.  The final
assertion is Theorem~\ref{thm:tree-reconstruction}.
\end{proof}

Subalgebras generated by characteristic classes also have substantial
precedent.  On homogeneous spaces, Postnikov, B. Shapiro, and M. Shapiro
studied algebras generated by curvature forms \cite{PostnikovShapiroShapiro};
related zonotopal algebras are controlled by the graphic matroid
\cite{NenashevZonotopal}.
Filtered Postnikov--Shapiro and external bizonotopal algebras can even be
complete graph invariants
\cite{NenashevShapiro,KirillovNenashevShapiroVaintrob}.

The degree-four part has an additional elementary precedent.  The classical
edge algebra of $G$ is the monomial subalgebra
\[
  \Bbbk[x_ux_v:uv\in E(G)]\subseteq\Bbbk[x_v:v\in V(G)]
\]
\cite{Katzman}.  The subalgebra generated only by the classes $b_{uv}=h_uh_v$
is an image of this edge algebra in the polynomial sector of
$H_G^{(n)}$.  The classes $a_{uv}=h_u+h_v$, the truncation and diagonal
relations, and the simultaneous dependence of the ambient cohomology ring on
$G$ distinguish $A_G^{(n)}$ from the classical edge ring and from the
matroidal zonotopal algebras.

For comparison, the Stanley--Reisner ring of the independence complex,
\[
  \operatorname{SR}(G)
  =
  \Bbbk[x_v:v\in V(G)]/(x_ux_v:uv\in E(G)),
\]
is a covariant graph functor under $x_v\mapsto x_{f(v)}$, and its abstract
algebra type determines $G$ \cite{BrunsGubeladze}.  Here adjacency is imposed
directly by the defining monomial ideal.  By contrast, $A_G^{(n)}$ is
extracted from the cohomology of a geometric configuration space, and its
relations arise from cup products of edge-span Chern classes.

Other graph-indexed constructions include graph cohomology categorifying the
chromatic polynomial \cite{HelmeGuizonRong} and Hom complexes used for
coloring obstructions \cite{BabsonKozlov}.

\section{Further questions}

The construction leads to a focused collection of questions.
\begin{enumerate}
  \item How does $A_G^{(n)}$ vary or stabilize with the ambient dimension
  $n$, and what information is lost in passing from $H_G^{(n)}$ to
  $A_G^{(n)}$?
  \item Can the deletion--contraction cospan of
  Section~\ref{subsec:deletion-contraction} be enhanced to an exact or derived
  relationship among $A_{G\setminus e}^{(n)}$, $A_G^{(n)}$, and
  $A_{G/e}^{(n)}$, or yield recurrences for their algebraic invariants?  How
  does $\mathcal A_n$ interact with graph products?
  \item Can the conductor description in
  Corollary~\ref{cor:bipartite-preimage} be used to describe the
  multiplication or presentations of $A_T^{(n)}$ for general $n$?  In
  particular, does $A_T^{(n)}$ determine $T$ for every $n\geq 4$, as it does
  for $n=3$?  Which broader graph classes can be reconstructed from
  $A_G^{(3)}$?
  \item How far can the graphical model and the intrinsic presentation
  \eqref{eq:canonical-presentation} be scaled beyond five vertices without
  computing every multiplication table explicitly?
  \item What does the kernel of
  $B_G^{(n)}\twoheadrightarrow A_G^{(n)}$ record about the collision boundary
  of the picture variety?
\end{enumerate}

\section*{Acknowledgments}

This disclosure follows the recommendations of the Leiden Declaration on
Artificial Intelligence and Mathematics \cite{LeidenDeclaration}. OpenAI's
Codex was used extensively to develop the construction and arguments, draft and
revise the manuscript, conduct preliminary literature searches, and implement
and test the software. Anthropic's Claude was used for discussion and editorial
feedback. The author verified the final arguments, references, and computational
claims and assumes full responsibility for the contents.

\bibliographystyle{amsplain}
\bibliography{edge_span_chern_algebras_of_graphical_configuration_spaces}

\providecommand{\bysame}{\leavevmode\hbox to3em{\hrulefill}\thinspace}
\providecommand{\MR}{\relax\ifhmode\unskip\space\fi MR }
\providecommand{\MRhref}[2]{%
  \href{http://www.ams.org/mathscinet-getitem?mr=#1}{#2}
}
\providecommand{\href}[2]{#2}
\begin{thebibliography}{10}

\bibitem{AlisteDeMierZamora}
Jos{\'e} Aliste-Prieto, Anna de~Mier, and Jos{\'e} Zamora, \emph{On trees with
  the same restricted {$U$}-polynomial and the {Prouhet--Tarry--Escott}
  problem}, Discrete Math. \textbf{340} (2017), no.~6, 1435--1441,
  \href{https://doi.org/10.1016/j.disc.2016.09.019}{doi:10.1016/j.disc.2016.09.019}.

\bibitem{LeidenDeclaration}
Jarod Alper et~al., \emph{Leiden declaration on artificial intelligence and
  mathematics}, Zenodo, June 2026, Declaration,
  \href{https://doi.org/10.5281/zenodo.20302944}{doi:10.5281/zenodo.20302944}.

\bibitem{BabsonKozlov}
Eric Babson and Dmitry~N. Kozlov, \emph{Complexes of graph homomorphisms},
  Israel J. Math. \textbf{152} (2006), no.~1, 285--312,
  \href{https://doi.org/10.1007/BF02771988}{doi:10.1007/BF02771988}.

\bibitem{BaranovskySazdanovic}
Vladimir Baranovsky and Radmila Sazdanovi{\'c}, \emph{Graph homology and graph
  configuration spaces}, J. Homotopy Relat. Struct. \textbf{7} (2012), no.~2,
  223--235,
  \href{https://doi.org/10.1007/s40062-012-0006-3}{doi:10.1007/s40062-012-0006-3}.

\bibitem{BerceanuMarklPapadima}
Barbu Berceanu, Martin Markl, and {\c S}tefan Papadima, \emph{Multiplicative
  models for configuration spaces of algebraic varieties}, Topology \textbf{44}
  (2005), no.~2, 415--440,
  \href{https://doi.org/10.1016/j.top.2004.10.002}{doi:10.1016/j.top.2004.10.002}.

\bibitem{BokstedtMinuz}
Marcel B{\"o}kstedt and Erica Minuz, \emph{Graph cohomologies and rational
  homotopy type of configuration spaces},
  \href{https://arxiv.org/abs/1904.01452v3}{arXiv:1904.01452v3}, 2019, revised
  2020.

\bibitem{BrunsGubeladze}
Winfried Bruns and Joseph Gubeladze, \emph{Combinatorial invariance of
  {Stanley--Reisner} rings}, Georgian Math. J. \textbf{3} (1996), no.~4,
  315--318,
  \href{https://doi.org/10.1515/GMJ.1996.315}{doi:10.1515/GMJ.1996.315}.

\bibitem{ChaudharyGordon}
Sharad Chaudhary and Gary Gordon, \emph{Tutte polynomials for trees}, J. Graph
  Theory \textbf{15} (1991), no.~3, 317--331,
  \href{https://doi.org/10.1002/jgt.3190150308}{doi:10.1002/jgt.3190150308}.

\bibitem{CrowleyEtAl}
Colin Crowley, Galen Dorpalen-Barry, Andr{\'e} Henriques, and Nicholas
  Proudfoot, \emph{The geometry of zonotopal algebras {I}: cohomology of
  graphical configuration spaces},
  \href{https://arxiv.org/abs/2502.12768v2}{arXiv:2502.12768v2}, 2025.

\bibitem{EastwoodHuggett}
Michael Eastwood and Stephen Huggett, \emph{Euler characteristics and chromatic
  polynomials}, European J. Combin. \textbf{28} (2007), no.~6, 1553--1560,
  \href{https://doi.org/10.1016/j.ejc.2006.09.005}{doi:10.1016/j.ejc.2006.09.005}.

\bibitem{HelmeGuizonRong}
Laure Helme-Guizon and Yongwu Rong, \emph{A categorification for the chromatic
  polynomial}, Algebr. Geom. Topol. \textbf{5} (2005), no.~4, 1365--1388,
  \href{https://doi.org/10.2140/agt.2005.5.1365}{doi:10.2140/agt.2005.5.1365}.

\bibitem{HoekstraMendoza}
Teresa~I. Hoekstra-Mendoza, \emph{Recovering trees from the cohomology ring of
  their configuration spaces},
  \href{https://arxiv.org/abs/2312.16646v1}{arXiv:2312.16646v1}, 2023.

\bibitem{Katzman}
Mordechai Katzman, \emph{Bipartite graphs whose edge algebras are complete
  intersections}, J. Algebra \textbf{220} (1999), no.~2, 519--530,
  \href{https://doi.org/10.1006/jabr.1999.7919}{doi:10.1006/jabr.1999.7919}.

\bibitem{KhoroshkinLyskov}
Anton Khoroshkin and Denis Lyskov, \emph{Graphical configuration spaces,
  contractads and formality},
  \href{https://arxiv.org/abs/2509.21255v4}{arXiv:2509.21255v4}, 2026.

\bibitem{KirillovNenashevShapiroVaintrob}
Anatol Kirillov, Gleb Nenashev, Boris Shapiro, and Arkady Vaintrob,
  \emph{Bizonotopal graphical algebras}, Algebraic Combinatorics \textbf{9}
  (2026), no.~3, 831--863,
  \href{https://doi.org/10.5802/alco.488}{doi:10.5802/alco.488}.

\bibitem{Kriz}
Igor K{\v r}{\'\i}{\v z}, \emph{On the rational homotopy type of configuration
  spaces}, Ann. of Math. (2) \textbf{139} (1994), no.~2, 227--237,
  \href{https://doi.org/10.2307/2946581}{doi:10.2307/2946581}.

\bibitem{LiuTreePolynomial}
Pengyu Liu, \emph{A tree distinguishing polynomial}, Discrete Appl. Math.
  \textbf{288} (2021), 1--8,
  \href{https://doi.org/10.1016/j.dam.2020.08.019}{doi:10.1016/j.dam.2020.08.019}.

\bibitem{MartinGeometry}
Jeremy~L. Martin, \emph{Geometry of graph varieties}, Trans. Amer. Math. Soc.
  \textbf{355} (2003), no.~10, 4151--4169,
  \href{https://doi.org/10.1090/S0002-9947-03-03321-X}{doi:10.1090/S0002-9947-03-03321-X}.

\bibitem{MartinTopology}
\bysame, \emph{On the topology of graph picture spaces}, Adv. Math.
  \textbf{191} (2005), no.~2, 312--338,
  \href{https://doi.org/10.1016/j.aim.2004.03.010}{doi:10.1016/j.aim.2004.03.010}.

\bibitem{NenashevZonotopal}
Gleb Nenashev, \emph{Classification of external zonotopal algebras}, Electron.
  J. Combin. \textbf{26} (2019), no.~1, Paper No. P1.32, 10 pp.,
  \href{https://doi.org/10.37236/8299}{doi:10.37236/8299}.

\bibitem{NenashevShapiro}
Gleb~V. Nenashev and Boris Shapiro, \emph{``{K}-theoretic'' analog of
  {Postnikov--Shapiro} algebra distinguishes graphs}, J. Combin. Theory Ser. A
  \textbf{148} (2017), 316--332,
  \href{https://doi.org/10.1016/j.jcta.2017.01.001}{doi:10.1016/j.jcta.2017.01.001}.

\bibitem{NishimuraKneser}
Yusaku Nishimura, \emph{The {Kneser} chromatic function distinguishes trees},
  Discrete Math. \textbf{349} (2026), no.~6, Paper No. 115009,
  \href{https://doi.org/10.1016/j.disc.2026.115009}{doi:10.1016/j.disc.2026.115009}.

\bibitem{NobleWelsh}
Steven~D. Noble and Dominic J.~A. Welsh, \emph{A weighted graph polynomial from
  chromatic invariants of knots}, Ann. Inst. Fourier (Grenoble) \textbf{49}
  (1999), no.~3, 1057--1087,
  \href{https://doi.org/10.5802/aif.1706}{doi:10.5802/aif.1706}.

\bibitem{OrlikSolomon}
Peter Orlik and Louis Solomon, \emph{Combinatorics and topology of complements
  of hyperplanes}, Invent. Math. \textbf{56} (1980), no.~2, 167--189,
  \href{https://doi.org/10.1007/BF01392549}{doi:10.1007/BF01392549}.

\bibitem{PostnikovShapiroShapiro}
Alexander Postnikov, Boris Shapiro, and Mikhail Shapiro, \emph{Algebras of
  curvature forms on homogeneous manifolds}, Differential topology,
  infinite-dimensional {Lie} algebras, and applications (Alexander Astashkevich
  and Serge Tabachnikov, eds.), Amer. Math. Soc. Transl. Ser. 2, vol. 194,
  Amer. Math. Soc., Providence, RI, 1999,
  \href{https://doi.org/10.1090/trans2/194/10}{doi:10.1090/trans2/194/10},
  pp.~227--235.

\bibitem{StanleyChromaticSymmetric}
Richard~P. Stanley, \emph{A symmetric function generalization of the chromatic
  polynomial of a graph}, Adv. Math. \textbf{111} (1995), no.~1, 166--194,
  \href{https://doi.org/10.1006/aima.1995.1020}{doi:10.1006/aima.1995.1020}.

\bibitem{Totaro}
Burt Totaro, \emph{Configuration spaces of algebraic varieties}, Topology
  \textbf{35} (1996), no.~4, 1057--1067,
  \href{https://doi.org/10.1016/0040-9383(95)00058-5}{doi:10.1016/0040-9383(95)00058-5}.

\bibitem{WaltersGraphFunctorSoftware}
Jackson Walters, \emph{{jacksonwalters/graph-functor: Graph Functor v1.0.3:
  Sage-Free Reproducibility Fixes}}, Zenodo, 2026, Software release~v1.0.3,
  \href{https://doi.org/10.5281/zenodo.21538969}{doi:10.5281/zenodo.21538969}.

\bibitem{Whitney}
Hassler Whitney, \emph{Congruent graphs and the connectivity of graphs}, Amer.
  J. Math. \textbf{54} (1932), no.~1, 150--168,
  \href{https://doi.org/10.2307/2371086}{doi:10.2307/2371086}.

\bibitem{ZakharovThesis}
Alexander Zakharov, \emph{Rational homotopy type of complements of smooth
  arrangements}, Ph.D. thesis, Sorbonne Universit{\'e}, 2025, Thesis
  no.~2025SORUS160,
  \href{https://doi.org/10.70675/e899d468z522ez4e32za5a1z2359ce488423}{doi:10.70675/e899d468z522ez4e32za5a1z2359ce488423}.

\bibitem{Zakharov}
\bysame, \emph{Rational homotopy type of complements of submanifold
  arrangements}, \href{https://arxiv.org/abs/2211.05033v3}{arXiv:2211.05033v3},
  2026.

\end{thebibliography}

\end{document}